\documentclass[conference]{IEEEtran}
\IEEEoverridecommandlockouts

\usepackage{amsmath,amssymb,amsthm}
\usepackage{graphicx}
\usepackage{amsfonts}
\usepackage{algorithm}
\usepackage{algpseudocode}
\usepackage{url}
\usepackage{cite}
\usepackage{textcomp}
\usepackage{xcolor}

\def\BibTeX{{\rm B\kern-.05em{\sc i\kern-.025em b}\kern-.08em
    T\kern-.1667em\lower.7ex\hbox{E}\kern-.125emX}}

\newcommand{\reals}{\mathbb{R}}

\newtheorem{theorem}{Theorem}
\newtheorem{assumption}{Assumption}

\DeclareMathOperator{\Range}{Range}
\DeclareMathOperator{\Span}{Span}
\DeclareMathOperator{\Expect}{\mathbb{E}}

\begin{document}

%

\title{AdaSub: Stochastic Optimization Using Second-Order Information in Low-Dimensional Subspaces\thanks{This work was supported by the Novo Nordisk Foundation under grant number NNF20OC0061894.}
}

\iftrue 
\author{\IEEEauthorblockN{João Victor Galvão da Mata}
\IEEEauthorblockA{\textit{Department of Applied Mathematics and Computer Science} \\
\textit{Technical University of Denmark}\\
jogal@dtu.dk \\
ORCID 0009-0009-8251-6476}
\and
\IEEEauthorblockN{Martin S.~Andersen}
\IEEEauthorblockA{\textit{Department of Applied Mathematics and Computer Science} \\
\textit{Technical University of Denmark}\\
mskan@dtu.dk\\
ORCID 0000-0002-4654-3946}
}
\fi 

\maketitle              
\begin{abstract}
We introduce AdaSub, a stochastic optimization algorithm that computes a search direction based on second-order information in a low-dimensional subspace that is defined adaptively based on available current and past information. Compared to first-order methods, second-order methods exhibit better convergence characteristics, but the need to compute the Hessian matrix at each iteration results in excessive computational expenses, making them impractical. To address this issue, our approach enables the management of computational expenses and algorithm efficiency by enabling the selection of the subspace dimension for the search. Our code is freely available on GitHub, and our preliminary numerical results demonstrate that AdaSub surpasses popular stochastic optimizers in terms of time and number of iterations required to reach a given accuracy.
\end{abstract}

\section{Introduction}
The optimization of neural network (NN) models is often a challenging task due to the high-dimensional and non-convex nature of many such problems. First-order optimization methods such as stochastic gradient descent (SGD) and its many variants are widely used as the main optimization methods for training NN models \cite{StochasticSubgradientMethod,AdaGrad,adam}, but they often require extensive hyperparameter tuning to reach an acceptable training loss \cite{AdaHessian} and only rely on noisy first-order information, resulting in slow convergence and poor generalization for ill-conditioned training problems \cite{QnewtonForMachineLearning}.

Hyperparameters, including learning rate, weight decay, and momentum parameters, are essential for the training process and require careful tuning. However, tuning these parameters can be costly in terms of both computational resources and human effort. In addition, a poor choice of optimizer and/or parameters can lead to significant performance degradation, highlighting the limitations of first-order methods \cite{QnewtonForMachineLearning,AdaHessian}.

Optimization methods that incorporate second-order information can more effectively capture and leverage the curvature of the loss function. As a result, they have the potential to find better search directions. However, this improved performance generally comes at the cost of additional computational resources.

This paper presents ``AdaSub,'' a stochastic optimizer that utilizes second-order information in a low-dimensional subspace without requiring the explicit computation of the Hessian matrix. As a result, our method avoids significant increases in the computational cost. Our approach is inspired by recent methods such as AdaHessian \cite{AdaHessian} and SketchySGD \cite{SketchySGD}, which access second-order information through Hessian-vector products. Specifically, given a twice continuously differentiable function $f \colon \reals^n \to \reals$ and vectors $v \in \reals^n$ and $w \in \reals^n$, we have that
\begin{equation}
\label{hessian_vec_prod}
   \left.\frac{d \nabla f(w+tv)}{dt}\right|_{t=0} = \frac{\partial( \nabla f(w)^T v)}{\partial w} = \nabla^2 f(w) v,
\end{equation}
which can be computed without forming $\nabla^2 f(w)$ by leveraing automatic differentiation. SketchySGD uses such Hessian-vector products to compute a diagonal plus low-rank approximation to $\nabla^2 f(w)$ using sketching techniques, and AdaHessian uses similar ideas to compute a diagonal approximation to $\nabla^2 f(w)$. In contrast, we use Hessian-vector products to compute a second-order approximation of the objective function $f$ on a subspace that is defined adaptively based on a current (stochastic) gradient and a small set of past (stochastic) gradients rather than random directions. This paper presents the following contributions:
\begin{enumerate}
\item We propose a new stochastic optimization algorithm that generalizes SGD by incorporating second-order information in a low-dimensional subspace. The dimension of the subspace is a parameter of the algorithm.
\item We provide a basic convergence analysis of our method in the special case where the objective function is strongly convex and smooth.
\item We demonstrate the effectiveness of AdaSub through numerical experiments based on various learning tasks. Our method matches or outperforms state-of-the-art methods in terms of the number of iterations.
\item Our preliminary implementation is available on GitHub\footnote{Available at: \url{https://github.com/Jvictormata/adasub}}.
\end{enumerate}

\section{Adaptive Subspace Minimization}

The basic idea behind our method is to first compute a second-order approximation of a twice continuously differentiable objective function $f \colon \reals^n \to \reals$ on an $m$-dimensional affine set. Given a vector $w \in \reals^n$ and a matrix $V \in \reals^{n \times m}$ with orthonormal columns and with $m \ll n$, the restriction of $f$ to the affine set $w + \Range(V)$ may be parameterized as $h(z) = f(w + Vz)$ where $z \in \reals^m$. Using the chain rule, we can express the second-order Taylor expansion of $h$ around $0$ as
\begin{align}
\notag \tilde h(z) &= h(0) + \nabla h(0)^Tz + \frac12 z^T\nabla^2h(0)z \\ 
&= f(w) + g^TVz + \frac12  z^TV^TH Vz,
\end{align}
where we use the shorthand notation $g = \nabla f(w)$ and $H = \nabla^2 f(w)$. In the special case where $\nabla^2h(z) \succ 0$, the unique minimizer of $\tilde h$ is given by $z^\star = -(V^TH V)^{-1}V^Tg$, which motivates an update of the form 
\begin{equation}
\label{update_rule_eq}
w^+ \gets w + \eta \, Vz^\star = w - \eta \, V (V^THV)^{-1} V^T g,
\end{equation} 
where $\eta >0$ is a step size. This corresponds to a Newton step for the restriction of $f$. The computation of $z^\star$ requires $m$ Hessian-vector products to form $HV$, and hence the cost of computing $z^\star$ grows as $O(mC + m^3)$ if $C$ is the cost of a Hessian-vector product. This is much lower than the cost of computing a Newton step in the full $n$-dimensional space, which is generally too costly when $n$ is huge. 

In the general case where $H$ is not necessarily positive definite, an update of the form 
\begin{equation}
    \label{update_rule_eq_general}
    w \gets w - \eta \, V (V^THV + \alpha I)^{-1} V^T g,
\end{equation} 
may be viewed as a regularized Newton step for the restriction of $f$, i.e., we find $z^\star$ by minimizing $\tilde h(z) + \alpha/2 \|z\|_2^2$. This update is guaranteed to be a descent direction if $V^Tg \neq 0$ and $\alpha$ is chosen such that $V^THV + \alpha I \succ 0$. We note that
\begin{align*}
 (\alpha I + V V^THVV^T)^{-1}& = \\
 V(V^THV + &\alpha I)^{-1}V^T +  \alpha^{-1} (I-VV^T),
\end{align*}
so if $g \in \Range(V)$, which implies that $(I-VV^T)g = 0$, then \eqref{update_rule_eq_general} is equivalent to
\[ w \gets w - \eta \, (\alpha I + V V^THVV^T)^{-1} g, \]
where $V V^THVV^T$ is a low-rank approximation to $H$. Moreover, this update is similar to the update rule used in SketchySGD \cite{SketchySGD}, which is of the form 
\[ w \gets w - \eta \, (\rho I + \hat H)^{-1} g,  \] 
where $\hat H$ is a low-rank approximation to $H$ that is obtained by means of a randomized Nystr{\"o}m approximation.

\subsection{Adaptive Subspace Selection}

We now consider an iteration of the form
\begin{align}\label{e-adasub-basic}
    w_{k+1} = w_k - \eta_k\, V_k(V_k^TH_kV_k + \alpha_k I)^{-1}V_k^Tg_k
\end{align}
for $k = 0,1,2,\ldots$ and where $w_0 \in \reals^n$ is an initial guess, $g_k = \nabla f(w_k)$, $H_k = \nabla^2f(w_k)$, and $V_k \in \reals^{n \times m}$ has orthonormal columns. We will construct $V_k$ such that $g_k \in \Range(V_k)$, and we choose $\alpha_k$ such that $V_k^TH_kV_k + \alpha_k I \succ 0$. This yields a descent direction, whenever $\nabla f(w_k) \neq 0$, since
\begin{align*} \left. \frac{d}{d\eta} f(w_k - \eta V_k(V_k^TH_kV_k + \alpha_k I)^{-1}V_k^Tg_k)\right|_{\eta=0} \\=  - \|V_k^Tg_k\|_{(V_k^TH_kV_k + \alpha_k I)^{-1}}^2 < 0, \end{align*} 
where $\|x\|^2_W = x^T W x$ denotes the weighted norm of $x$ induced by the symmetric positive definite matrix $W$ of order $n$. For example, $V_k$ can be constructed such that
\begin{equation*}
    \Range(V_k) \supseteq \Span(\nabla f(w_{k}),\nabla f(w_{k-1}), \dots, \nabla f(w_{k-m+1})),
\end{equation*}
and $\alpha_k$ can be chosen as $\alpha_k = \max(0, \rho - \lambda_{\min}(V_k^TH_kV) )$ where $\rho > 0$ is a given threshold, which implies that $\lambda_{\min}(V_k^TH_kV_k + \alpha_k I) \geq \rho$.

\subsection{Block-Diagonal Approximation for Deep Neural Networks}\label{ss-blockdiag-approx}
For deep NNs, the computation of Hessian-vector products can be quite expensive, and to reduce this cost, we may consider an approximation in which the Hessian is computed layer for layer. This corresponds to a block-diagonal approximation to the Hessian. 

Given a NN with $p$ layers, we partition the parameter vector $w_k$ as $w_k = (w_{k,1},\ldots,w_{k,p})$ such that $w_{k,l}$ contains all the parameters associated with layer $l$. Similarly, if we let $f \colon \reals^n \to \reals$ denote the loss function and partition $\nabla f(w_k)$ as $g_k = (g_{k,1},\ldots,g_{k,p})$ where $g_{k,l}$ is the gradient of $f(w_k)$ with respect to $w_{k,l}$, then
\begin{equation*}
    H_{k,l} v_l = \frac{\partial( g_{k,l}^T v_l)}{\partial w_{k,l}}
\end{equation*}
is the product of the Hessian associated with layer $l$ and a vector $v_l$. The resulting approximation to $H_k$ can be expressed as
\begin{equation*}
    \hat{H}_k = \begin{bmatrix}
    H_{k,1} & 0 & \dots & 0  \\
    0 & H_{k,2} & \dots & 0 \\
    \vdots & \vdots & \ddots & \vdots \\
    0 & 0 & \dots & H_{k,p} \\
  \end{bmatrix},
\end{equation*}
and a matrix-vector product $\hat H_k v$ can be computed efficiently using automatic differentiation. Algorithm~1 outlines a variant of the AdaSub method, which we will analyze in a stochastic setting in the next section.
\begin{algorithm}[h]
\caption{AdaSub with block-diagonal Hessian approx.}\label{alg:cap}
\begin{algorithmic}
\Require subspace dimension $m$, learning rate $\eta>0$, threshold $\rho>0$
\State $k \gets 0$
\Repeat
\For{$l = 1$ to $p$} 

    \If{$k = 0$}
        \State $G_l \gets g_{k,l}$
    \Else
        \State $G_l \gets \begin{bmatrix} G_l & g_{k,l} \end{bmatrix}$
        \If{$k \geq m$} 
            \State Delete the first column of $G_l$
        \EndIf
    \EndIf

    \State Compute orthonormal basis $V_{k,l}$ for $\Range(G_l)$
    \State Compute $Y \gets H_{k,l}V_{k,l}$
    \State Form $Y^TV_{k,l}$ 
    \State Compute spectral decomposition $Y^TV_{k,l} = U\Lambda U^T$
    \State $\alpha \gets \max(0,\rho - \min(\Lambda))$
    \State $w_{k+1,l} \gets w_{k,l} - \eta_k \, V_{k,l} U(\Lambda + \alpha I)^{-1}U^TV_{k,l}^T g_{k,l}$
    \EndFor
    \State $k \gets k + 1$
\Until{stopping criteria met}
\end{algorithmic}
\end{algorithm}

\section{Analysis}
We now analyze the iteration \eqref{e-adasub-basic} in a stochastic setting under the assumption that the function $f$ is twice continuously differentiable and of the form
\begin{align}
    f(w) = \frac{1}{N} \sum_{i=1}^N f_i(w) + \phi(w),
\end{align}
where $\phi \colon \reals^n \to \reals$ is a regularization function.
Given two random subsets $S_k$ and $\tilde S_k$ of $\{1,\ldots,N\}$, we define a minibatch gradient and Hessian at iteration $k$ as
\begin{align*}
    g_k &= \frac{1}{|S_k|} \sum_{i=1}^{|S_k|} \nabla f_i(w_k)+\nabla \phi(w_k),\\
     H_k &= \frac{1}{|\tilde S_k|} \sum_{i=1}^{|\tilde S_k|} \nabla^2 f_i(w_k) + \nabla^2 \phi(w_k).
\end{align*}
Furthermore, we will make the following assumptions:
\begin{assumption}\label{strong_convex_assumption}
    There exist positive constants $\mu$ and $L$ such that for all $w \in \reals^n$,
    \begin{equation*}
         \mu I  \preceq \nabla^2 f(w) \preceq  L I,
    \end{equation*} 
    i.e., $f$ is $\mu$-strongly convex and $L$-smooth.
\end{assumption}
\begin{assumption}\label{stochastic_assumption}
    The stochastic gradients are unbiased, i.e., $\Expect[\,g_k\,|\, w_k] = \nabla f(w_k)$,
    the stochastic Hessians satisfy $\mu I \preceq H_k \preceq L I$, and  
    $g_k$ and $H_k$ are independent random variables.
\end{assumption}
\begin{assumption}\label{assumption_bounded_sg}
    The stochastic gradients are bounded in the sense that \cite{stochasticGradientBounded}
    \begin{equation}\label{assumption_extra_eq}
        \Expect[\,\| g_k \|_2^2\,|\, w_k \,] \leq 2 \xi L (f(w_k)-f(w^\star)) + \sigma^2
    \end{equation} 
    where $\sigma^2 = \Expect[\,\| g(w^\star) \|_2^2\,]$ is the gradient noise of $f$ at the optimum and $\xi$ is a positive constant such that $\Expect[\,\| g_k \|_2^2\,] \leq \xi \| \nabla f(w_k) \|_2^2$. The inequality \eqref{assumption_extra_eq} is known as the ``weak growth condition'' \cite{assumptionWeakGrowth} in the special case where $\sigma^2 = 0$.
\end{assumption}

Assumption~\ref{strong_convex_assumption} implies that for all $w$ and $w'$,
\begin{equation}
\label{l_smooth}
f(w') \leq f(w) + \nabla f(w)^T (w'- w) + \frac{L}{2} \| w'- w \|_2^2,
\end{equation} 
and strong convexity implies that
\begin{equation}
\label{pl_ineq}
\| \nabla f(w) \|_2^2 \geq  2 \mu (f(w)-f(w^\star)),
\end{equation} 
which is known as the Polyak--Lojaciewicz (PL) inequality \cite{plInequality}. We note that the independence of $g_k$ and $H_k$ in Assumption~\ref{stochastic_assumption} can be ensured by letting the subsets $S_k$ and $\tilde S_k$ be independent.  

\begin{theorem} \label{thm-upper-bound}
    Let $f \colon \reals^n \to \reals$ be a twice continuously differentiable function, instate Assumptions~\ref{strong_convex_assumption}--\ref{assumption_bounded_sg}, and let $\kappa = L/\mu$. The iteration \eqref{e-adasub-basic} with the step size $\eta_k = \eta \in (0, \kappa^{-3}\xi^{-1}]$ then satisfies
    \begin{align}\label{e-lemma-ineq}
        \Expect[\,f(w_{k})\,]-f(w^\star) \leq \left( 1\!-\! \frac{\eta}{\kappa}  \right)^k (f(w_0)-f(w^\star)) + \frac{\kappa^2\eta}{2\mu} \sigma^2.
    \end{align}
\end{theorem}
\begin{proof}
    From the iteration \eqref{update_rule_eq}, we have that $w_{k+1}-w_k = -\eta \tilde H_k^+ g_k$ where $\tilde H_k^+ = V_k (V_k^TH_kV_k)^{-1}V_k^T$ denotes the Moore--Penrose pseudo inverse of $\tilde H_k = V_k (V_k^TH_kV_k) V_k^T$. Moreover, using the inequality \eqref{l_smooth}, we have
    \begin{align*}
        f(w_{k+1}) &\leq  f(w_{k}) - \eta \, \nabla f(w_{k})^T \tilde H_k^+ g_k  + \frac{L\eta^2}{2} \| \tilde H_k^+ g_k \|_2^2 \\
        &\leq f(w_{k}) - \eta \, \nabla f(w_{k})^T \tilde H_k^+ g_k  + \frac{L\eta^2}{2} \| \tilde H_k^+ \|_2^2 \|g_k \|_2^2.
    \end{align*}
    Assumption~\ref{stochastic_assumption} implies that
    $\| \tilde H_k^+ \|_2 = \|(V_k^TH_kV_k)^{-1}\|_2 \leq 1/\mu$, which leads to
    \begin{align}
       f(w_{k+1}) &\leq f(w_{k}) - \eta \, \nabla f(w_{k})^T \tilde H_k^+ g_k  + \frac{L\eta^2}{2\mu^2} \|g_k \|_2^2.
    \end{align}
    Taking the expectation and using Assumptions~\ref{stochastic_assumption}--\ref{assumption_bounded_sg}, we find that
    \begin{align*}
        \Expect[\,f(w_{k+1})\,] &\leq f(w_{k}) - \eta \, \nabla f(w_k)^T \Expect[\,\tilde H_k^+\,] \Expect[\, g_k\,] \\
        & \qquad  + \frac{L\eta^2}{2\mu^2}  \Expect[\,\|g_k \|_2^2\,]\\
        &\leq f(w_{k}) - \frac{\eta}{L} \, \| \nabla f(w_{k})\|_2^2  + \frac{L\eta^2}{2\mu^2}  \Expect[\,\|g_k \|_2^2\,],
    \end{align*}
    where the last inequality follows from Assumption~\ref{stochastic_assumption} by noting that $\|\Expect[\, \tilde H_k^+ \,] \|_2 \geq 1/L$.
    Next, using the PL inequality \eqref{pl_ineq} and Assumption~\ref{assumption_bounded_sg}, we get
    \begin{align*}
      &  \Expect[\,f(w_{k+1})\,]-f(w^\star) \leq \\
       &\qquad \left(1 - \frac{2\eta\mu}{L} + \frac{L^2\eta^2\xi}{\mu^2} \right) (f(w_k)-f(w^\star)) + \frac{L\eta^2}{2\mu^2} \sigma^2.
    \end{align*}
    Letting $\eta \leq \mu^3/(L^3\xi)$, we see that
        \[ 1 - \frac{2\eta\mu}{L} + \frac{L^2\eta^2\xi}{\mu^2} = 1 - \frac{2\eta \mu}{L}\left(1- \frac{L^3\eta\xi}{2\mu^3}\right) \leq 1-\frac{\eta \mu}{L} = 1-\frac{\eta}{\kappa},\]
    and hence,
    \begin{align*}
        \Expect[\,f(w_{k+1})\,]-f(w^\star) \leq \left(1\!-\!\frac{\eta}{\kappa} \right) (f(w_k)-f(w^\star)) + \frac{\kappa \eta^2}{2\mu} \sigma^2.
    \end{align*}
    Finally, we arrive at \eqref{e-lemma-ineq} by recursively taking the expectation over all steps.
\end{proof}
A direct consequence of Theorem~\ref{thm-upper-bound} is that
$\Expect[\,f(w_{k}) \,] - f(w^\star) \leq \epsilon$
provided that the learning rate $\eta$ is sufficiently small and $k$ is sufficiently large. 

Although analyzing the non-convex case is important, it is a complex problem that requires further investigation. Thus, we will defer the analysis of the non-convex case to future work.

\section{Numerical Experiments}

To assess the practical performance of the AdaSub method, we have conducted four experiments based on different settings, including convex and non-convex problems and different architectures of NNs. In the first experiment, we apply the method to a logistic regression problem using a Python implementation of the method based on the JAX \cite{jax} library. The remaining experiments utilize NNs, for which we implement a PyTorch \cite{pytorch} optimizer that incorporates the block-diagonal Hessian approximation described in \S\ref{ss-blockdiag-approx}. We also use the ReLU activation function, despite its lack of twice continuous differentiability. In all our experiments, we use $\rho = 0.02$ as the lower bound for the eigenvalues of $V_k^T H_k V_k + \alpha_k I$.

We compare our method to widely used methods such as Adam \cite{adam}, AdaGrad \cite{AdaGrad}, and SGD in all experiments. For the experiments involving a NN, we also compare to AdaHessian \cite{AdaHessian}. We note that we were unable to obtain meaningful results with SketchySDG's default parameter choices, and therefore, we defer a fair comparison with SketchySDG to future work.

The learning rates for all the experiments were taken from papers with similar numerical experiments \cite{AdaHessian,SketchySGD} or manually chosen to ensure the best performance of each optimizer and a fair comparison.

\subsection{Logistic Regression}
In our first experiment, we trained a logistic regression model to classify if a given image features a cat. We used a data set with 1200 cat images from Kaggle \cite{cat_faces} and 1200 non-cat images of random objects (such as food, cars, houses, etc) from the "other" category of the Linnaeus 5 dataset \cite{linnaeus5}. Each image has $64 \times 64$ pixels, and we used 80\% of the images for training and 20\% for testing. We used a batch size of 192 images, $\eta = 0.2$, and a subspace of dimension $m=10$. With Adam/AdaGrad/SGD, the learning rates were tuned manually and set to be 0.005/0.001/0.2, and with Adam, we used the parameters $\beta_1 = 0.9$ and $\beta_2 = 0.999$.

The ensure a lower bound for the stochastic Hessian we used a weight decay with regularization parameter $\lambda = 10^{-4}$ in this experiment. Figure~\ref{result_convex_experiment} shows the results from this experiment and demonstrates AdaSub's advantage over the other methods in terms of both training loss and test accuracy.
 \begin{figure*}[htb]
    \centering
\includegraphics[width=\textwidth]{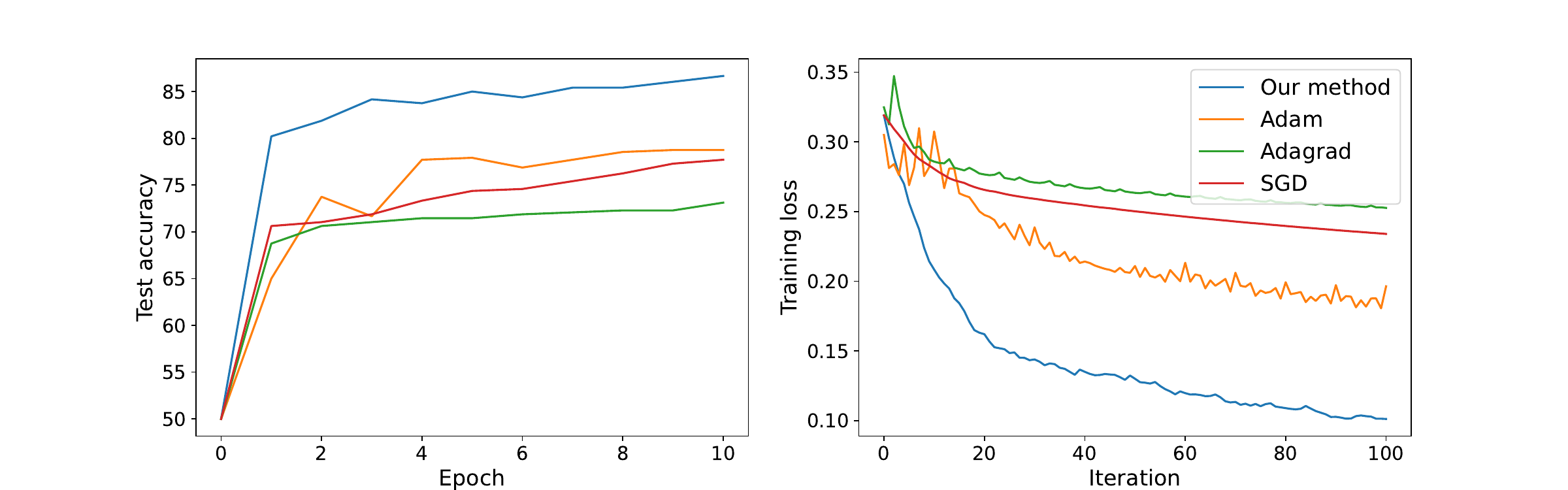}
\caption{Accuracy of the logistic regression model in the test set per epoch and value of the cost function in the training set per training step.} \label{result_convex_experiment}
\end{figure*}

\subsection{Fully Connected Neural Network}
Next, to test our method in a fully connected NN, we built a shallow fully connected NN and trained it on the MNIST \cite{mnist} dataset. This dataset consists of 70,000 grayscale images of handwritten numbers. We used 60,000 images as the training set and the remaining 10,000 images as the test set, and we used a batch size of 1000 images. Each image is represented as a vector of length 784 (corresponding to 28x28 pixels). We used a network with a single hidden layer with 100 units and a ReLU activation function, and as the output layer, we used a softmax activation function with 10 outputs. For AdaSub, we used $\eta= 0.1$ and a subspace of dimension $m=40$. For AdaHessian/Adam/AdaGrad/SGD, the learning rates were tuned manually and set to be 0.05/0.001/0.01/0.1. For AdaHessian/Adam, we used $\beta_1 = 0.9$ and $\beta_2 = 0.999$. We note that this is a nonconvex problem, and hence the analysis from the previous section does not apply. 

Figure~\ref{result_fully_connected} shows, for all the methods, the accuracy in the test set after each epoch and the cost function for the current minibatch at each iteration. AdaSub presented an advantage over the other methods in terms of both training loss and test accuracy.
\begin{figure*}[htbp]
    \centering
\includegraphics[width=\textwidth]{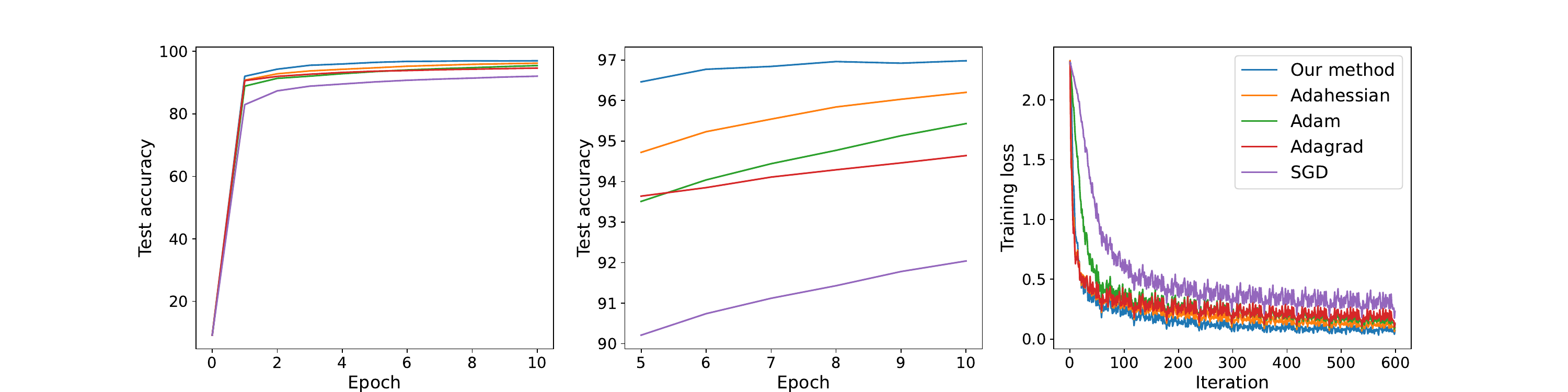}
\caption{Accuracy of the fully connected network in the test set per epoch and value of the cost function in the training set per training step.} \label{result_fully_connected}
\end{figure*}

\subsection{Convolutional Neural Network}
We have also applied our method to the problem of training a convolutional neural network (CNN) based on the CIFAR-10 \cite{cifar10} dataset. This dataset consists of 60,000 32x32 color images divided into 10 different classes. The classes are airplanes, cars, birds, cats, deer, dogs, frogs, horses, ships, and trucks, and each class has the same number of examples. We used 50,000 images as the training set, 10,000 images as the test set, and we used a batch size of 400 images.

The input to the CNN is a color image with dimensions 32x32x3. The network begins with a convolutional layer that has six output channels and uses 5x5x3 dimensional kernels. After the convolutional layer, a 2x2 max pooling layer with a stride of 2 is applied. This process is repeated with a second convolutional layer that has 16 output channels and uses 5x5x6 dimensional kernels, followed by another 2x2 max pooling layer with a stride of 2. The output from this step is a tensor with dimensions of 5x5x16. This tensor is flattened into a vector of length 400, which is then fed into a fully connected network with two hidden layers. The first hidden layer has 120 units, and the second has 84 units. The output layer has 10 units and uses a softmax activation function. All the layers except the output layer use a ReLU activation function.

In this experiment, we used $\eta = 0.15$ and a subspace of dimension $m=20$ in our method. For AdaHessian/Adam/AdaGrad/SGD, the learning rates were tuned and set to be 0.1/0.001/0.01/0.4. For AdaHessian/Adam, we used $\beta_1 = 0.9$ and $\beta_2 = 0.999$. 

Figure~\ref{result_conv} shows, for all the methods, the accuracy in the test set after each epoch and the cost function for the current minibatch at each iteration. Comparing these results with the ones from the previous experiment, we can observe that even for a more complex problem with a smaller adaptive subspace, AdaSub was able to outperform the other methods in terms of both training loss and test accuracy.
\begin{figure*}[htbp]
    \centering
\includegraphics[width=\textwidth]{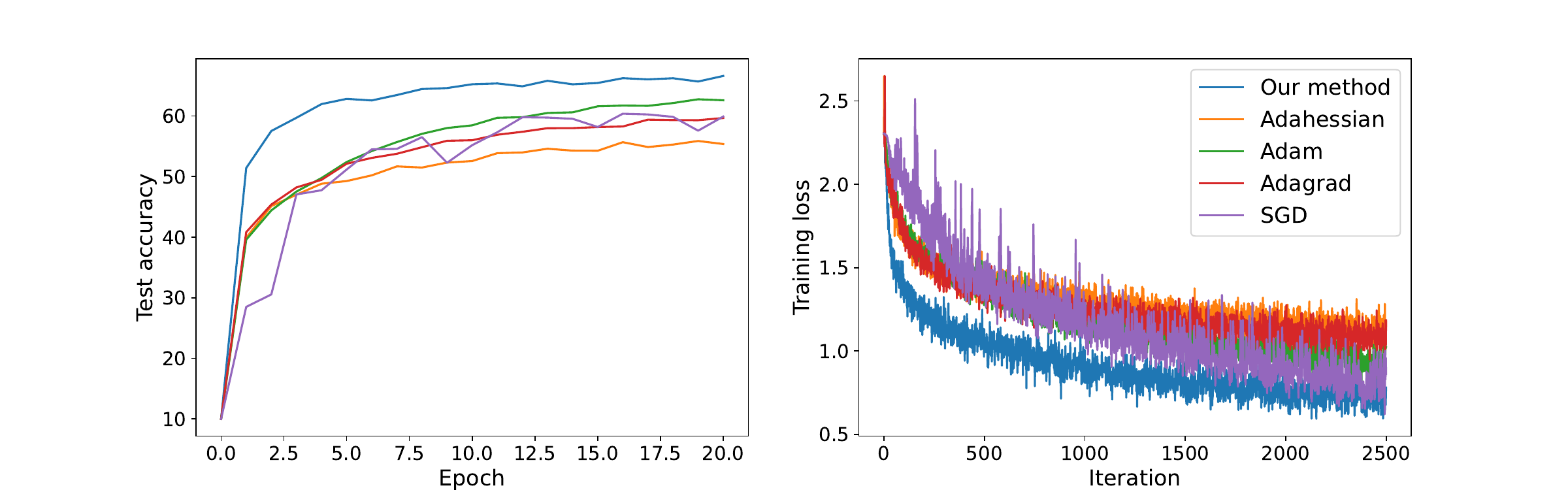}
\caption{Accuracy of the CNN in the test set per epoch and value of the cost function in the training set per training step.} \label{result_conv}
\end{figure*}

\subsection{ResNet20}
To test our method for the training of deeper NNs, we applied it to ResNet20 \cite{resnet} using the CIFAR-10 dataset. We used $\eta = 0.15$ and a batch size of 400 images. Since ResNet20 is a NN with approximately 270,000 parameters, we chose to use a two-dimensional subspace ($m=2$) for this experiment to limit the computational cost. Thus, we only need to perform two (approximate) Hessian-vector products at each iteration. For AdaHessian/Adam/AdaGrad/SGD, the learning rates were tuned and set to 0.2/0.001/0.1/0.1. For AdaHessian/Adam, we used $\beta_1 = 0.9$ and $\beta_2 = 0.999$.

Figure \ref{result_resnet} shows, for all the methods, the accuracy in the test set and the cost function for the current minibatch after each epoch. Comparing these results with the last experiments, we observe that our method still outperforms SDG/Adam/Adagrad and is comparable to AdaHessian despite the small subspace dimension.
\begin{figure*}[htbp]
    \centering
\includegraphics[width=\textwidth]{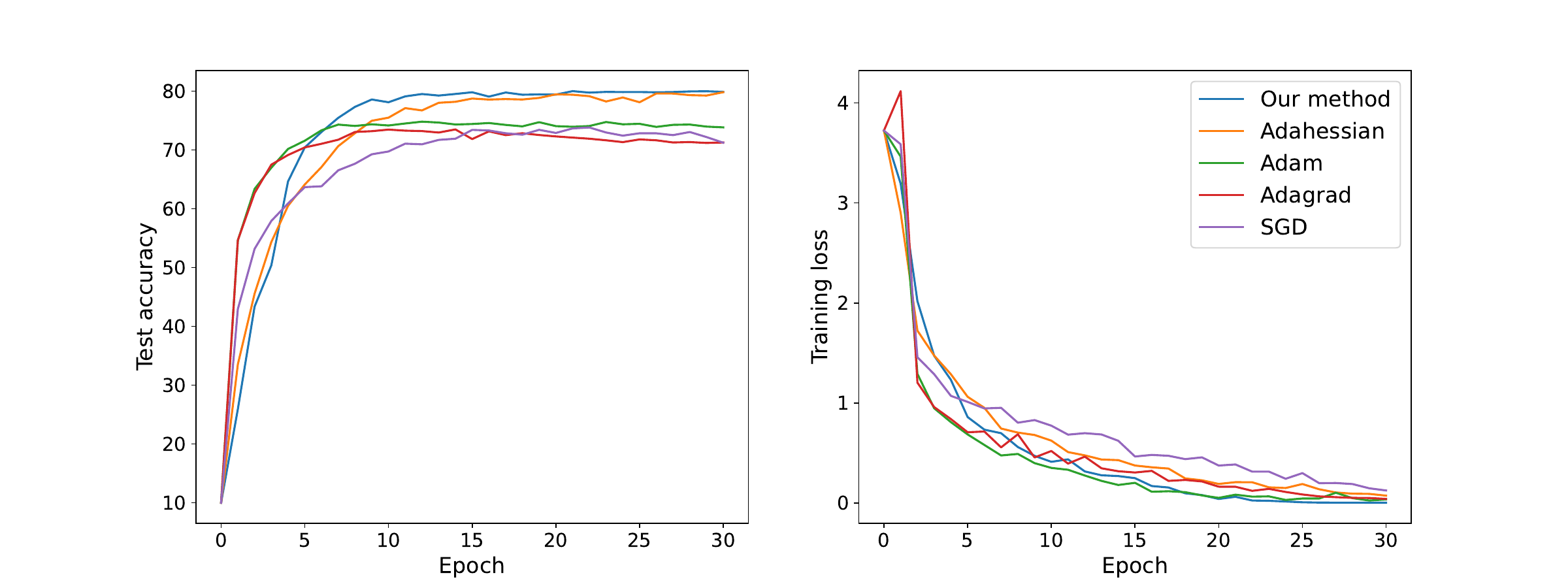}
\caption{Accuracy of ResNet20 in the test set and value of the cost function in the training set per epoch.} \label{result_resnet}
\end{figure*}

\subsection{Time Comparison}
The cost of one  minibatch iteration using AdaSub or AdaHessian is higher than that of Adam, Adagrad, and SGD because of the added cost of Hessian-vector products. Thus, to compare the runtime, we repeated the first experiment (logistic regression), letting each optimizer run for 10 seconds. 

Figure \ref{result_time} shows the test accuracy and the training loss as a function of the runtime. From the test accuracy plot, we can see that the test accuracy obtained after 10 seconds using Adam is matched by AdaSub after approximately 1.4 seconds. Moreover, in terms of training loss, AdaSub outperforms Adam after approximately 1 second.
\begin{figure*}[htb]
    \centering
\includegraphics[width=\textwidth]{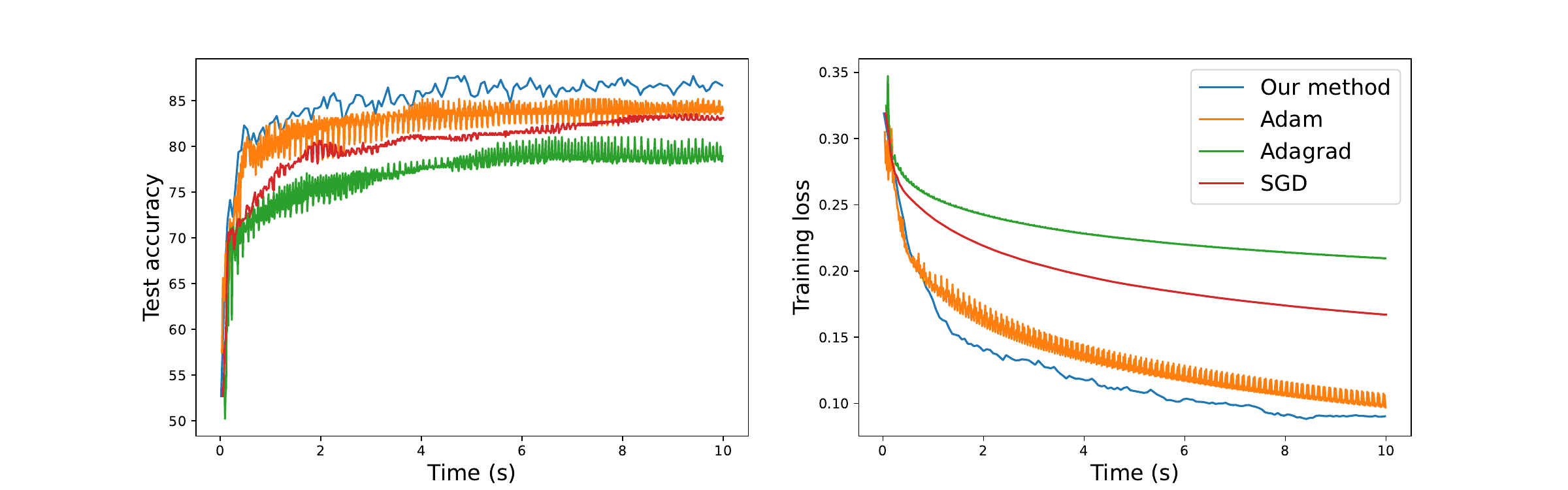}
\caption{Test accuracy and training loss as a function of time.} \label{result_time}
\end{figure*}

\section{Conclusions and Future Work}

We have presented AdaSub, a new stochastic optimization method that restricts the search direction at each iteration to a low-dimensional subspace of the previous gradients. Our method achieves a better balance between accuracy and computational efficiency by incorporating second-order information into the optimization process while limiting the dimension of the subspace. We tested the method on both convex and non-convex problems, and in both settings, it outperformed several state-of-the-art optimization methods. These results suggest that our proposed algorithm can be a promising approach for optimizing complex and high-dimensional models in a range of applications within machine learning. Moreover, our preliminary experiments suggest that AdaSub is also relatively robust to the choice of algorithmic parameters such as the number of past directions and the step size.

The current version of the algorithm is a first attempt at incorporating second-order information into the optimization process without significantly increasing the cost per iteration. The construction of the subspace and the practical implications of its dimension need to be further explored. For example, the construction of $V_k$ may be based on additional information such as previous search directions in addition to previous gradients. Moreover, a more thorough analysis of the convergence rate of the algorithm is needed under less restrictive and more realistic assumptions, and additional numerical experiments are needed to test the performance of the algorithm on a wider range of problems. In particular, we would like to test the algorithm on larger datasets and deeper NNs.  

%
%

\IEEEtriggeratref{15}

\begin{thebibliography}{10}

\bibitem{QnewtonForMachineLearning}
Albert~S. Berahas, Majid Jahani, Peter Richtárik, and Martin Takáč.
\newblock Quasi-{N}ewton methods for machine learning: Forget the past, just sample, 2019.

\bibitem{jax}
James Bradbury, Roy Frostig, Peter Hawkins, Matthew~James Johnson, Chris Leary, Dougal Maclaurin, George Necula, Adam Paszke, Jake Vander{P}las, Skye Wanderman-{M}ilne, and Qiao Zhang.
\newblock {JAX}: composable transformations of {P}ython+{N}um{P}y programs, 2018.

\bibitem{linnaeus5}
Giorgi Chaladze and Levan Kalatozishvili.
\newblock Linnaeus 5 dataset for machine learning, 2017.

\bibitem{StochasticSubgradientMethod}
Damek Davis, Dmitriy Drusvyatskiy, Sham Kakade, and Jason~D. Lee.
\newblock Stochastic subgradient method converges on tame functions, 2018.

\bibitem{AdaGrad}
John Duchi, Elad Hazan, and Yoram Singer.
\newblock Adaptive subgradient methods for online learning and stochastic optimization.
\newblock {\em Journal of Machine Learning Research}, 12(61):2121--2159, 2011.

\bibitem{SketchySGD}
Zachary Frangella, Pratik Rathore, Shipu Zhao, and Madeleine Udell.
\newblock Sketchy{SGD}: Reliable stochastic optimization via robust curvature estimates, 2022.

\bibitem{cat_faces}
Spandan Ghosh.
\newblock Cats faces 64x64 (for generative models), Feb 2019.

\bibitem{stochasticGradientBounded}
Robert~Mansel Gower, Nicolas Loizou, Xun Qian, Alibek Sailanbayev, Egor Shulgin, and Peter Richt{\'a}rik.
\newblock {SGD}: General analysis and improved rates.
\newblock In Kamalika Chaudhuri and Ruslan Salakhutdinov, editors, {\em Proceedings of the 36th International Conference on Machine Learning}, volume~97 of {\em Proceedings of Machine Learning Research}, pages 5200--5209. PMLR, 09--15 Jun 2019.

\bibitem{resnet}
Kaiming He, Xiangyu Zhang, Shaoqing Ren, and Jian Sun.
\newblock Deep residual learning for image recognition.
\newblock In {\em 2016 IEEE Conference on Computer Vision and Pattern Recognition (CVPR)}, pages 770--778, 2016.

\bibitem{plInequality}
Hamed Karimi, Julie Nutini, and Mark Schmidt.
\newblock Linear convergence of gradient and proximal-gradient methods under the polyak-Łojasiewicz condition, 2016.

\bibitem{adam}
Diederik~P. Kingma and Jimmy Ba.
\newblock Adam: A method for stochastic optimization, 2014.

\bibitem{cifar10}
Alex Krizhevsky and Geoffrey Hinton.
\newblock Learning multiple layers of features from tiny images.
\newblock Technical Report~0, University of Toronto, Toronto, Ontario, 2009.

\bibitem{mnist}
Yann LeCun and Corinna Cortes.
\newblock {MNIST} handwritten digit database.
\newblock 2010.

\bibitem{pytorch}
Adam Paszke, Sam Gross, Francisco Massa, Adam Lerer, James Bradbury, Gregory Chanan, Trevor Killeen, Zeming Lin, Natalia Gimelshein, Luca Antiga, Alban Desmaison, Andreas Kopf, Edward Yang, Zachary DeVito, Martin Raison, Alykhan Tejani, Sasank Chilamkurthy, Benoit Steiner, Lu~Fang, Junjie Bai, and Soumith Chintala.
\newblock Pytorch: An imperative style, high-performance deep learning library.
\newblock In H.~Wallach, H.~Larochelle, A.~Beygelzimer, F.~d\textquotesingle Alch\'{e}-Buc, E.~Fox, and R.~Garnett, editors, {\em Advances in Neural Information Processing Systems}, volume~32. Curran Associates, Inc., 2019.

\bibitem{assumptionWeakGrowth}
Sharan Vaswani, Francis~R. Bach, and Mark Schmidt.
\newblock Fast and faster convergence of {SGD} for over-parameterized models and an accelerated perceptron.
\newblock {\em CoRR}, abs/1810.07288, 2018.

\bibitem{AdaHessian}
Zhewei Yao, Amir Gholami, Sheng Shen, Kurt Keutzer, and Michael~W. Mahoney.
\newblock {ADAHESSIAN:} an adaptive second order optimizer for machine learning.
\newblock {\em CoRR}, abs/2006.00719, 2020.

\end{thebibliography}

\end{document}